\documentclass[12pt, reqno]{amsart}
\usepackage{amsmath, amssymb, amsfonts, amsthm, amscd}
\usepackage{manfnt}
\usepackage{mathtools}

\usepackage{float}
\usepackage{fullpage}
\usepackage{url}
\usepackage[all]{xy}
   \SelectTips{cm}{10}
\usepackage{setspace}
\usepackage{txfonts}
\usepackage{enumitem}
\usepackage[mathscr]{euscript}
\numberwithin{equation}{section}

\newtheorem{theorem}{Theorem}[section]
\newtheorem{corollary}{Corollary}[theorem]
\newtheorem{lemma}[theorem]{Lemma}
\newtheorem{prop}[theorem]{Proposition}
\newtheorem{defn}[theorem]{Definition}
\newtheorem*{remark}{Remark}

\newcommand{\Z}{\mathbb{Z}}
\newcommand{\Zi}{\mathbb{Z}[i]}
\newcommand{\Q}{\mathbb{Q}}
\newcommand{\Qi}{\mathbb{Q}[i]}
\newcommand{\R}{\mathbb{R}}
\newcommand{\bp}{\textbf{p}}

\begin{document}

\thanks{This paper is a culmination of a year of work with many people. I would like to thank Stefan Patrikis of the University of Utah Department of Mathematics for guiding me through this research and for cultivating my love of mathematics. My weekly discussions with him on of a variety of subtopics in number theory have been insightful and engaging. I also thank Sean McAfee, Christian Klevdal, Marcus Robinson, and Allechar Serrano L\'opez for guiding me through the beginnings of number theory at the Summer Mathematics Program for High School Students at the University of Utah. I'd like to thank various undergraduate and graduate students for reviewing this manuscript. I would especially like to thank my brother, who recently passed away, for inspiring me to study mathematics. He pushed me to go beyond what I was learning in the classroom. I am extremely grateful that I have done so today. When asked the practical applications of his mathematical work, my brother replied, "That's a silly question. It's beautiful." This is the spirit I will carry with me for the rest of my life, and I dedicate this paper to him.}

\title{The Frequency of Elliptic Curves Over $\mathbb{Q}[i]$ With Fixed Torsion}
\author{Alan Zhao}
\date{August 2019}
\email{asz2115@columbia.edu}

\maketitle

\begin{abstract}
Mazur\textsc{\char13}s Theorem states that there are precisely 15 possibilities for the torsion subgroup of an elliptic curve defined over the rational numbers. It was previously shown by Harron and Snowden that the number of isomorphism classes of elliptic curves of height up to $X$ that have a specific torsion subgroup $G$ is on the order of $X^{1/{d(G)}}$, for some positive $d(G)$ depending on $G$. We compute $d(G)$ for these groups over $\Qi$. Furthermore, in a collection of recent papers it was proven that there are 9 more possibilities for the torsion subgroup in the base field $\Qi$. We compute the value of $d(G)$ for these new groups. 
\end{abstract}
\tableofcontents

\section{Background on Elliptic Curves}
Before stating the objectives of this paper, I'd like to provide the reader with some background on elliptic curves. Let $K$ be a field. For the purposes of our work, an elliptic curve $E/K$ is the projective algebraic curve associated to a Weierstrass equation as in (1.2), where $a,b \in K$. We assume $\text{char}(K) \neq 2, 3$, and so by affine transformations we may reduce the representation of an elliptic curve $E$ to
\begin{equation}
    E_{aff}: y^2 = x^3 + ax + b
    \label{WEaffine}
\end{equation}
which yields the projective homogeneous representation
\begin{equation}
    E_{proj}: Y^2Z = X^3 + aXZ^2 + bZ^3.
    \label{WEprojective}
\end{equation}
We assume $4a^3+27b^2 \neq 0$, which guarantees a non-singular elliptic curve. Finally, define $\mathcal{O} = [0, 1, 0]$ to be the "point at infinity" on $E$.

\begin{defn}
Denote by $E(K)$ the set of $K$-rational points on $E$.
\end{defn}

\subsection{The Group Law on an Elliptic Curve}
We will now define the group law $\oplus$ on $E(K)$. Let $P, Q \in E(K)$, where $P$ and $Q$ are not necessarily distinct. By B\'ezout's Theorem, the line $\ell_{proj}$ through $P$ and $Q$ must intersect $E$ exactly three times, counting multiplicity. Therefore, $\ell_{proj}$ must intersect $E$ at precisely one more point. Denote this point as $R$.


\begin{defn}
Let $-R = R \oplus \mathcal{O}$ and define $P \oplus Q = -R$. 

\end{defn}
From the discussion so far, $E(K)$ is closed under $\oplus$,  $\mathcal{O}$ is the identity element, and $P \oplus -P = \mathcal{O}$, so every element has an inverse. Commutativity follows from the fact that two points define a line. It remains to prove associativity, which follows from a more involved application of Bezout's Theorem (see \cite{IK}).

Thus we have that $E(K)$ is an abelian group. The Mordell-Weil Theorem asserts that when $K$ is a number field, $E(K)$ is finitely generated and thus is isomorphic to $\Z^r \times E(K)_{tors}$ for some positive integer $r$ and finite subgroup $E(K)_{tors}$ (called the ``torsion subgroup"). This subgroup will be central to the counting problem of \S \S 2-4.

Since $\Qi$ is a number field, we may apply all the theory developed in this section to an elliptic curve $E/\Qi$. We now proceed to count elliptic curves $E$ that have a fixed torsion subgroup.

\section{Statement of the Main Theorem}
The goal of this paper will be to extend the results of Harron and Snowden's paper, which gives the relative frequencies for which isomorphism classes of elliptic curves over $\Q$ with prescribed torsion subgroup $G$ will appear up to a given height. We generalize these results to elliptic curves over $\Qi$.

We begin with the following definitions and notation:
\begin{itemize}
    \item $N(z)$ denotes the norm of $z \in \Qi $.
    \item $X \in \R^+$.
    \item $p$ denotes a prime of $\Z$, and $\bp$ denotes a prime of $\Zi$.
    \item A Weierstrass representation $E:y^2 = x^3 + Ax + B$ ($A, B \in \Zi $) of an elliptic curve $E/\Qi$ is \textit{minimal} if and only if $\gcd(A^3, B^2)$ is not divisible by any 12th power of a non-unit in $\Zi$.
    \item The \textit{height} of the representation $E:y^2 = x^3 + Ax + B$ is defined as $\max(N(A)^3, N(B)^2)$.
\end{itemize}
Any elliptic curve $E/\Qi$ is isomorphic to a curve represented by a minimal Weierstrass equation
\begin{equation}
    E : y^2 = x^3 + Ax + B
\end{equation}
because $\Zi$ is a unique factorization domain, and so if $\bp^{12} \mid \gcd(A^3, B^2)$, then $\bp^4 \mid A$ and $\bp^6 \mid B$, and by a change of variable $E$ is also isomorphic to $y^2 = x^3 + \bp^{-4}Ax+\bp^{-6}Bx$. Furthermore, note that one may make the change of variable $(x,y) \to (x/i^2, y/i^3)$ on $E$. This yields the elliptic curve $E_1:y^2=x^3+Ax - B$, with $E \simeq E_1$. Thus, for $B = 0$ there will be one minimal Weierstrass representation, and for $B \neq 0$ there will be two isomorphic minimal Weierstrass representations. 


\begin{defn}
Let $N_G(X)$ be the number of isomorphism classes of elliptic curves of height less than $X$ with torsion subgroup isomorphic to $G$. 
\end{defn}
\begin{remark}
This quantity must be finite because there are only finitely many elements in $\Zi$ that have a fixed integral norm.
\end{remark}
In \cite{RHAS} it was shown that for each group $G$ in Mazur's theorem, there is an explicit constant $d(G)$ such that elliptic curves over $\mathbb{Q}$ with torsion subgroup $G$ satisfy the asymptotic
\begin{equation}
    \frac{1}{d(G)} =  \lim_{X \to \infty} \frac{\text{log }N_G(X)}{\text{log } X}
\end{equation}
\cite{RHAS}. We first partially generalize the methods used in Harron and Snowden's paper to $\Qi$. The precise possibilities for new torsion subgroups over $\Qi$ are given by the following theorem:
\begin{theorem}
(Kenku, Momose, Kamienny) \cite{KM}\cite{KUMO} Let $E$ be an elliptic curve over $ \mathbb{Q}[i]$. Then, $E(\Qi)_{tors}$ must be isomorphic to one of the following groups:
\begin{itemize}
\item $\mathbb{Z}/M\mathbb{Z}, 1 \leq M \leq 16$ or $ M = 18$
\item $\mathbb{Z}/2\mathbb{Z} \times \mathbb{Z}/2M\mathbb{Z}, 1 \leq M \leq 6 $
\item $\mathbb{Z}/4\mathbb{Z} \times \mathbb{Z}/4\mathbb{Z}$.
\end{itemize}
\end{theorem}

Finally, we note that for every group $G$ in Theorem 2.3 (except for $\Z/2\Z \times \Z/2\Z$, which is dealt with in \S 5), there exists a universal family of elliptic curves $\mathcal{E}$ equipped with a subgroup isomorphic to $G$ \cite{FPR}. The case of $G = \Z/2\Z \times \Z/2\Z$ is exceptional and will be dealt with in \S 5. We may parameterize the coefficients of the Weierstrass forms with two functions. For the groups in Theorem 2.2, these functions are polynomials in one variable or two variables. For a fixed $G$ we have the universal family $\mathcal{E}_t:y^2 = x^3+f(t)x+g(t)$, or, for the two variable case, $\mathcal{E}_{s,t}:y^2 = x^3+f(s, t)x+g(s,t)$ (where the points $(s,t) \in \Qi^2$ lie on a curve $C \hookrightarrow \mathbb{P}^2$) such that $G \subseteq \mathcal{E}_t(\Qi)_{tors}$ or $G \subseteq \mathcal{E}_{s,t}(\Qi)_{tors}$. And any $E$ such that $G \subseteq E(\Qi)_{tors}$ is isomorphic to some curve of these two forms.

Summarizing the results of \cite{FPR}, the three cases that will be dealt with in \S 4 are:
\begin{itemize}
    \item Case 1: $G = \Z/M\Z, \Z/2\Z \times \Z/2K\Z, \Z/4\Z \times \Z/4\Z$, $M \in [4,10]$, $M=12$, $2 \leq K \leq 4$, is parameterized by $\mathcal{E}_t$ and $t \in \Qi$.
    \item Case 2: $G = \Z/M\Z$, $M = 13, 16, 18$, is parameterized by $\mathcal{E}_{s,t}$ with $(s,t) \in C(\Qi)$, where $C$ is a plane curve of genus $>1$.
    \item Case 3: $G = \Z/M\Z, \Z/2\Z \times \Z/10\Z, \Z/2\Z \times \Z/12\Z $, $M = 11, 14, 15 $, is parameterized by $\mathcal{E}_{s, t}$ with $(s,t) \in C(\Qi)$, where $C$ has genus 1.
\end{itemize}

Our main theorem is the following, generalizing \cite[Theorem~1.1]{RHAS}:
\begin{theorem}
For torsion subgroups $G \neq \Z/M\Z$ ($M = 1, 2, 3$), we obtain values of $d(G)$ given by the table below.
\begin{table}[H]
\centering
  \begin{tabular}{ | l | l | l | l | l |
   l |}
    \hline
    $G$ & $ d(G) $ & $G$ & $d(G)$ & $G$ & $d(G)$ \\ \hline
    $\Z/4\Z$ & 4 &  $\Z/11\Z$ & $+\infty$ & $\Z/2\Z \times \Z/2\Z $ & 3 \\ \hline
    $\Z/5\Z $ & 6 &  $\Z/12\Z $ & 24 & $\Z/2\Z \times \Z/4\Z $ & 6 \\ \hline
    $\Z/6\Z$ & 6 &  $\Z/13\Z$ & $+\infty$ & $\Z/2\Z \times \Z/6\Z $ & 12 \\ \hline
    $\Z/7\Z$ & 12 &  $\Z/14\Z$ & $+\infty$ & $\Z/2\Z \times \Z/8\Z $ & 24 \\ \hline
    $\Z/8\Z$ & 12 &  $\Z/15\Z$ & $+\infty$ & $\Z/2\Z \times \Z/10\Z $ & $+\infty$ \\ \hline
    $\Z/9\Z$ & 18 &  $\Z/16\Z$ & $+\infty$ &$\Z/2\Z \times \Z/12\Z $ & $+\infty$ \\ \hline
    $\Z/10\Z$ & 18 & $\Z/18\Z $ & $+\infty$ & $\Z/4\Z \times \Z/4\Z $ & 12 \\ \hline
  \end{tabular}
  \end{table}
\end{theorem}


\section{Reduction of the Problem}
We extend the proofs used in \S \S 1-2 of \cite{RHAS} to $\Qi$, providing slight modifications. The structure of this section will be to trace through these first two sections.

\subsection{Reduction of the Problem}
Motivated by the discussion in \S 2, we have the following definition:
\begin{defn}
 $N^2_G(X) = \{(A,B) \in \Zi^2: E: y^2 = x^3+Ax+B \text{ is a minimal Weierstrass }$ 

\noindent $ \text{equation of height less than } $X$ \text{ and } E(\Qi)_{tors} = G \}$.
\end{defn}
We shall compute the quantity $N^2_G(X)$, which roughly counts the number of isomorphism classes of curves up to height $X$ with a torsion subgroup isomorphic to $G$. For the groups treated in \S 4.1, we have that $N_G(X) = \frac{1}{2}N^2_G(X) + O(1)$. See \S 4.1 for a proof of this. In \S \S 4.2-4.3, we in fact do not use the reductions provided in this section to treat the groups in Case 2 and Case 3. Then, for the groups treated in \S 4.1 by replacing $N_G(X)$ with $\frac{1}{2} N^2_G(X)+O(1)$ we do not change the value of the limit:
\begin{equation}
    \lim_{X \to \infty} \frac{\text{log }N_G(X)}{\text{log } X} = \lim_{X \to \infty} \frac{\text{log }\frac{1}{2}N^2_G(X)+O(1)}{\text{log } X} = \lim_{X \to \infty} \frac{\text{log }N^2_G(X)-\text{log }2 + O(1)}{\text{log } X} = \lim_{X \to \infty} \frac{\text{log }N^2_G(X)}{\text{log } X}.
\end{equation}
 As in \cite{RHAS}, let $N'_G(X)$ be the number of elliptic curves up to height  $X$  with a torsion subgroup containing $G$. We similarly define the quantity $N^{2'}_G(X)$:
 \begin{defn}
$N^2_G(X) = \{(A,B) \in \Zi^2: E: y^2 = x^3+Ax+B \text{ is a minimal Weierstrass }$ 

\noindent $ \text{equation of height less than } $X$ \text{ and } E(\Qi)_{tors} \subseteq G \}$. \end{defn}
We now prove the following theorem.
\begin{theorem}
For the groups $G$ and constants $d(G)$ listed in Theorem 2.3, there exist positive constants $K_1$ and $K_2$ such that $K_1X^{1/d(G)} \leq N'_G(X) \leq K_2X^{1/d(G)}$.
\end{theorem}

\begin{lemma}
Theorem 3.3 implies Theorem 2.3.
\end{lemma}

\begin{proof}
We have the following bounds:
\begin{equation}
    N_G^{2'}(X) - \sum_{G \subsetneq H} N_H^{2'}(X) \leq N_G^2(X) \leq N_G^{2'}(X).
\end{equation}
We have that $d(G) < d(H)$ when $G \subsetneq H$, so $N_G^2(X)/N_G^{2'}(X) \to 1$ as $X \to \infty$. 

The error term $O(1)$ becomes negligible as $X \to \infty$, so we have $N_G(X)/N_G'(X) \to 1$ as $X \to \infty$. Theorem 3.3 implies that $\frac{1}{d(G)} = \lim_{X \to \infty} \frac{\text{log }N_G(X)}{\text{log } X} $, proving Theorem 2.3.
\end{proof}

\section{Generalization of Harron and Snowden's Paper}
By (3.1), to prove Theorem 2.3 we may instead analyze the quantity $N_G^2(X)$.
\subsection{Case 1}
See page 3 above Theorem 2.3 for the list of groups in this case. We begin by proving the assertions of \S 3.1:
\begin{lemma}
$N^{2'}_G(X) = N^2_G(X) + O(1)$, where $O(1) \leq k \deg(g)$
\end{lemma}
\begin{proof}
We have $B = 0$ if and only if $g(t) = 0$. There are at most $\deg(g)$ such $t$. Finally, there are only finitely many $u$ such that $y^2 = x^3+u^4f(t)$ is minimal.
\end{proof}
Having reduced to computing $N^{2'}_G(X)$, we compute it in Proposition 4.2. We show that \cite[Proposition 2.1]{RHAS} continues to hold in $\Zi$. The following proposition implies Theorem 3.3, and so by Lemma 3.4 it implies Theorem 2.3 for the groups that fall into this case:
\begin{prop}
Let $ f, g \in \mathbb{Q}[t] $ be coprime polynomials of degrees $ r $ and $s$. Assume at least one of $r$ or $s$ is positive. Write 
\begin{center} 
max$(\frac{r}{4}, \frac{s}{6}) = \frac{n}{m}$
\end{center}
with $n$ and $m$ coprime. Assume $n=1$ or $m=1$. Let $S(X)$ be the set of pairs $(A,B) \in \mathbb{Z}[i]^2$ satisfying the following conditions: 

\begin{itemize}
\item $4A^3+27B^2 \neq 0$.
\item $\text{gcd}(A^3, B^2)$ is not divisible by any 12th power.
\item $N(A) < X^{1/3}$ and $N(B) < X^{1/2}$.
\item There exist $ u, t \in \mathbb{Q}[i]$ such that $A=u^4f(t)$ and $B=u^6g(t)$.
\end{itemize}
Then, there exist positive real constants $K_1$ and $K_2$ such that $K_1X^{(m+1)/12n} \leq \left| S(X) \right| \leq K_2X^{(m+1)/12n}$. Since $\left| S(X) \right| = N_G^{2'}(X)$, this proves Theorem 3.3.
\end{prop}
We will now proceed to prove Proposition 4.2, starting with the upper bound. We have the following data for the groups in Case 1:
\begin{table}[H]
\centering
  \begin{tabular}{ | l | l | l | l | l | l | }
    \hline
    $G$ & $r$ & $s$ & $n$ & $m$ & $12n/(m+1)$ \\ \hline
    $\mathbb{Z}/4\Z$ & $2$ & $3$ & $1$ & $2$ & $4$ \\ \hline
    $\mathbb{Z}/5\Z$ & $4$ & $6$ & $1$ & $1$ & $6$ \\ \hline
    $\mathbb{Z}/6\Z$ & $4$ & $6$ & $1$ & $1$ & $6$ \\ \hline
    $\mathbb{Z}/7\Z$ & $8$ & $12$ & $2$ & $1$ & $12$ \\ \hline
    $\mathbb{Z}/8\Z$ & $8$ & $12$ & $2$ & $1$ & $12$ \\ \hline
    $\mathbb{Z}/9\Z$ & $12$ & $18$ & $3$ & $1$ & $18$ \\ \hline
    $\mathbb{Z}/10\Z$ & $12$ & $18$ & $3$ & $1$ & $18$ \\ \hline
    $\mathbb{Z}/12\Z$ & $16$ & $24$ & $4$ & $1$ & $24$ \\ \hline
    $\mathbb{Z}/2\Z \times \mathbb{Z}/4\Z$ & $4$ & $6$ & $1$ & $1$ & $6$ \\ \hline
    $\mathbb{Z}/2\Z \times \mathbb{Z}/6\Z$ & $8$ & $12$ & $2$ & $1$ & $12$ \\ \hline   
    $\mathbb{Z}/2\Z \times \mathbb{Z}/8\Z$ & $16$ & $24$ & $4$ & $1$ & $24$ \\ \hline 
    $\mathbb{Z}/4\Z \times \mathbb{Z}/4\Z$ & $8$ & $12$ & $2$ & $1$ & $12$ \\ \hline 
  \end{tabular}
  \caption{Data for the universal elliptic curves.}
  \label{table: 1}
  \end{table}
The data were computed using the table \cite[\S 1.4, Table 2]{RHAS} and the computations of the degrees seen in Appendix A, derived from \cite{RHAS}.
\subsubsection{The Upper Bound}
Continue to let $f, g \in \Q [t]$ where $f$ and $g$ are coprime. Let $\overline{\mathbb{Q}}$ be a fixed algebraic closure of $\mathbb{Q}$, and extend $\left|\cdot\right|_{\bp}$ to $\overline{\mathbb{Q}}$. Let $\{\alpha_i\}$ be the roots of $f$ and $\{\beta_j\}$ the roots of $g$ in $\overline{\mathbb{Q}}$. We now have the following definition:
\begin{defn}
Let $S_1(X)$ denote the set of pairs $(u,t) \in \Qi^2$ such that $(A,B) = (u^4f(t), u^6g(t)) \in S(X)$.
\end{defn}

\begin{lemma}
Let $f, g \in \mathbb{Q}[t] $ be two coprime polynomials. For each prime $\textbf{p}$ of $\mathbb{Z}[i]$ (with $N(\bp) \leq \infty $), there exists a constant $c_{\textbf{p}}>0$ such that for all $t\in\mathbb{Q}[i]$,

\begin{equation}
\text{max}(\left|f(t)\right|_{\textbf{p}}, \left|g(t)\right|_{\textbf{p}}) \geq c_{\textbf{p}}.
\end{equation}

\noindent If $N(\textbf{p})$ is sufficiently large, one can take $c_{\bp}=1$.
\end{lemma}
\begin{remark}
We require the case of $N(\bp) = \infty$ for the proof of Lemma 4.7, where the standard absolute value assumes this case.
\end{remark}
\begin{proof}
We first analyze the roots of $f$ and $g$. Note that $\alpha_i \neq \beta_j$ for all $i$ and $j$ since $f$ and $g$ are coprime. 

Let $\delta = \min_{ij}(\left|\alpha_i - \beta_j\right|_{\bp})$. Let $\epsilon>0$ be such that $\left|f(t)\right|_{\bp} < \epsilon$ implies $\left|t-\alpha_i\right|_{\bp} < \delta/2$ for some $i$ and $\left| g(t) \right|_{\bp} < \epsilon $ implies $\left|t-\beta_j\right|_{\bp} < \delta/2$ for some $j$. To see why such an $\epsilon$ exists, consider the factorizations $f(t) = a\prod_{i = 1}^{r}(t-\alpha_i)$ and $g(t) = b\prod_{j = 1}^{s}(t-\beta_j)$ where $a,b \in \R$. If $\left|f(t)\right|_{\bp}, \left| g(t) \right|_{\bp} < \epsilon$, there exists indices $m$ and $n$ where $\left|t - \alpha_m \right|_{\bp} \leq (\epsilon/a)^{1/r}$ and $\left| t - \beta_n \right|_{\bp} \leq (\epsilon/b)^{1/s}$. We may choose $\epsilon$ such that $(\epsilon/a)^{1/r}, (\epsilon/b)^{1/s} < \delta/2$.

We proceed with proof by contradiction. If $\left|f(t)\right|_{\bp} < \epsilon$ and $\left|g(t)\right|_{\bp} < \epsilon$ then $\left|t-\alpha_i\right|_{\bp} < \delta/2$ and $\left|t-\beta_j\right|_{\bp} < \delta/2$ for some $i$ and $j$. However, this implies that $\left|\alpha_i-\beta_j\right|_{\bp} < \delta$ by the triangle inequality, a contradiction as $\delta $ is the minimal value of $\left|\alpha_i-\beta_j\right|_{\bp}$ for any $i$ and $j$. We must therefore have for all $t$ that either $\left|f(t)\right|_{\bp} \geq \epsilon$ or $\left|g(t)\right|_{\bp} \geq \epsilon$, so we can take $c_{\bp}=\epsilon$.

Now let $p$ be an integer prime with sufficiently large norm such that: (1) the coefficients of $f$ and $g$ are in $\mathbb{Z}_{p}$; (2) the leading coefficients of $f$ and $g$ are elements of $\Z_{p}^{\times}$ ; and (3) $\alpha_i-\beta_j$ are elements of $\overline{\Z_{p}}^{\times}$, for all $i$ and $j$.

Condition 3 gives us $\delta = 1$, since $\alpha_i-\beta_j$ then must satisfy $\left|\alpha_i-\beta_j\right|_{p} = p^0 = 1$.  If $\left|f(t)\right|_{p} < 1$ and $\left|g(t)\right|_{p} < 1$, then $\left|t-\alpha_i\right|_{p} < 1$ for some $i$, and $\left|t-\beta_j\right|_{p} < 1$ for some $j$. For this pair $(i,j)$, this implies $\left|\alpha_i-\beta_j\right|_{p} < 1 $, another contradiction, this time by the non-Archimedean triangle inequality that exists when working with the $p$-adic norm. That is, $ \left|x+y\right|_{p} \leq \text{max}(\left|x\right|_{p}, \left|y\right|_{p})$. Thus, we can take $c_{p} = 1$ for all primes $p$ with sufficiently large norm.
\end{proof}

We now prove a variant of \cite[Lemma 2.3]{RHAS}.

\begin{lemma}
For each Gaussian prime $\bp$ there exists a constant $C_{\bp}$ with the following property. Suppose $(u,t)\in S_1(X)$. Then

\begin{center}
$ \text{val}_{\bp}(u) = \epsilon + 
\begin{cases}
\lceil - \frac{n}{m} \text{val}_{\bp}(t)\rceil  & \text{$\text{val}_{\bp}(t) < 0$} \\
0 & \text{$\text{val}_{\bp}(t) \geq 0$}
\end{cases} $
\end{center}
where \textit{$\left|\epsilon\right| \leq C_{\bp}$. Furthermore, one can take $C_{\bp}=0$ for $N(\bp) \gg 0$.}
\end{lemma}

\begin{proof}
Let $B(x,r)$ be the open ball of radius $r$ centered at $x$. Fix an arbitrarily small constant $\delta > 0$ such that $B(\alpha_i, \delta) \cap B(\beta_j, \delta)$ is empty for all $i$ and $j$ and each $B(\alpha_i, \delta)$ and $B(\beta_j, \delta)$ contains at most one root of $f$ and $g$, respectively. Furthermore, suppose this $\delta$ satisfies that $t \in B(\alpha_i, \delta)$ implies that $\text{min}(\lfloor \frac{1}{4}\text{val}_{\bp}(f(t)) \rfloor, \lfloor \frac{1}{6}\text{val}_{\bp}(g(t)) \rfloor) = \lfloor \frac{1}{6}\text{val}_{\bp}(g(t)) \rfloor$ and similarly for $t \in B(\beta_j, \delta)$.

Before splitting into cases we make some general deductions. Suppose $ (u,t) \in S_1(X)$, and let $\bp$ be a prime. Since $A$ and $B$ are integral, we have $4 \text{val}_{\bp}(u) + \text{val}_{\bp}(f(t)) \geq 0 $ and $ 6 \text{val}_{\bp}(u) + \text{val}_{\bp}(g(t)) \geq 0$. Furthermore, $\text{val}_{\bp}(u) $ must be minimal subject to these inequalities, or else $\bp^{12}$ would divide $ \text{gcd}(A^3, B^2) $, a contradiction. We can thus write the following:
\begin{equation}
 \text{val}_{\bp}(u) =  \text{max}(\lceil -\frac{1}{4}\text{val}_{\bp}(f(t))\rceil, \lceil -\frac{1}{6}\text{val}_{\bp}(g(t)) \rceil).
 \end{equation}
By taking the negative of both sides of (4.2) we have:
\begin{equation}
-\text{val}_{\bp}(u) = \text{min}(\lfloor \frac{1}{4}\text{val}_{\bp}(f(t)) \rfloor, \lfloor \frac{1}{6}\text{val}_{\bp}(g(t)) \rfloor).
\end{equation}
We split into the following cases:
\begin{enumerate}
    \item \textbf{Case I:} $t \notin B(\alpha_i, \delta), B(\beta_j, \delta)$ for any $i$ and $j$.
    \item \textbf{Case II:} $t$ is in at most one of the balls.
\end{enumerate}
where $t \notin B(\alpha_i, \delta), B(\beta_j, \delta)$ for any $i$ and $j$ and where $t$ is in at most one of the balls. 

Assume now that $\text{val}_{\bp}(t) < 0$.


\textbf{Case I:}
We prove there exists a positive $K_1$ such that $\left|\text{val}_{\bp}(f(t)) - r \text{val}_{\bp}(t)\right| = \left| \text{val}_{\bp}(f(t)/t^r) \right| < K_1$ and $\left|\text{val}_{\bp}(g(t)) - s \text{val}_{\bp}(t)\right| = \left| \text{val}_{\bp}(g(t)/t^s) \right| < K_1$ for all such $t$. As stated in Proposition 4.2, the degrees of $f$ (resp. $g$) are $r$ (resp. $s$). Then,

\begin{equation}
\frac{f(t)}{t^r} = \sum_{k = 0}^{r} c_kt^{-k}
\end{equation}
and
\begin{equation}
\frac{g(t)}{t^s} = \sum_{j = 0}^{s} d_jt^{-j}
\end{equation}

\noindent Note that because $t \notin B(\alpha_i, \delta), B(\beta_j, \delta)$ for any $i$ and $j$, for $\text{val}_{\bp}(t) \leq v$ for some sufficiently negative $v$, we have that $\text{val}_{\bp}(f(t)/t^r) = \text{val}_{\bp}(c_0) $ and $\text{val}_{\bp}(g(t)/t^s) = \text{val}_{\bp}(d_0)$. For the case where $v < \text{val}_{\bp}(t) < 0$, the set of such $t$ is compact, and since $\text{val}_{\bp}(f(t)/t^r)$ and $\text{val}_{\bp}(g(t)/t^s)$ are both continuous (since $f(t), g(t) \neq 0$ on $\{t \in \Qi_{\bp}: v < \text{val}_{\bp}(t) < 0 \}$, they must be bounded. Choose $M$ such that $ \left| \text{val}_{\bp}(f(t)/t^r) \right|, \left| \text{val}_{\bp}(g(t)/t^s)\right| \leq M$. We may then take $K_1 = \max(M,\text{val}_{\bp}(c_0),\text{val}_{\bp}(d_0))$.

For $\text{val}_{\bp}(t) < 0$, we then have
\begin{center}
$\text{val}_{\bp}(u) = \epsilon + \text{max}(\lceil -\frac{r}{4} \text{val}_{\bp}(t) \rceil, \lceil -\frac{s}{6} \text{val}_{\bp}(t) \rceil)$,
\end{center}
where $\left|\epsilon\right| \leq K_2 $, for some constant $K_2$ (e.g., $K_2 = 1+\frac{n}{m}K_1$). Since $\text{val}_{\bp}(t) < 0$, we have
\begin{center}
$\text{max}(\lceil -\frac{r}{4} \text{val}_{\bp}(t) \rceil, \lceil -\frac{s}{6} \text{val}_{\bp}(t) \rceil)=\lceil-\frac{n}{m}\text{val}_{\bp}(t)\rceil$.
\end{center}

\textbf{Case II:}
Without loss of generality, assume $t \in B(\alpha_i, \delta)$. Here, $-\text{val}_{\bp}(u) = \lfloor \frac{1}{6}\text{val}_{\bp}(g(t)) \rfloor$. For all groups $G$ in Table 1, $\frac{r}{4} = \frac{s}{6}$, and so $-\text{val}_{\bp}(u) = \lfloor \frac{1}{6}\text{val}_{\bp}(g(t)) \rfloor = \epsilon + \lfloor \frac{n}{m}\text{val}_{\bp}(t) \rfloor$ where $\left| \epsilon \right| \leq C_2$ for some positive constant $C_2$.

Now we consider $\text{val}_{\bp}(t) \geq 0 $. Let $K_3$ be a constant such that $\text{min}(\text{val}_{\bp}(f(t)), \text{val}_{\bp}(g(t)) \leq K_3 $ for all such $t$, which exists by Lemma 4.4. We may use (4.3) to conclude that $-\text{val}_{\bp}(u) \leq K_4$ for an appropriate $K_4$ (the existence of $K_3$ demonstrates this). Let $K_5$ be so that $\text{val}_{\bp}(f(t)) \geq K_5 $ and $\text{val}_{\bp}(g(t)) \geq K_5 $ for all $t$ with $\text{val}_{\bp}(t) \geq 0 $. Using (4.3) again, we find $-\text{val}_{\bp}(u) \geq K_6$, for an appropriate $K_6$ since we have now a lower bound on $-\text{val}_{\bp}(u)$ with $K_5$. We find that $\left|\text{val}_{\bp}(u)\right| \leq K_7$ in this case, with $K_7 = \text{max}(K_4, K_6)$. 

Combined with the proof of the case $\text{val}_{\bp}(t) < 0$, we have proved the formula in the statement of the lemma, as we can take $C_{\bp}=\text{max}(C_1, C_2, K_2,K_7)$. 

Now suppose $p$ is a sufficiently large integer prime so that the coefficients of $f$ and $g$ are in $\mathbb{Z}_{p}$, the leading coefficients of $f$ and $g$ are elements of $(\Z_{p})^{\times}$, and that $c_{p}=1$ from Lemma 4.4. For $\text{val}_p(t) < 0$ we have $\text{val}_{p}(f(t)) = r\text{ val}_{p}(t)$ and $\text{val}_{p}(g(t)) = s\text{ val}_{p}(t)$, which shows that
\begin{center}
$\text{val}_{p}(u) = \text{max}(\lceil-\frac{r}{4} \text{val}_{p}(t) \rceil, \lceil-\frac{s}{6} \text{val}_{p}(t) \rceil) = \lceil-\frac{n}{m} \text{val}_{p}(t) \rceil $
\end{center}

For $\text{val}_{p}(t) \geq 0 $, we have $\text{val}_{p}(f(t)) \geq 0$ and $\text{val}_{p}(g(t)) \geq 0$, with at least one inequality being an equality (otherwise we contradict Lemma 4.4 with $c_{p}=1$). Thus, $\text{val}_{p}(u) = 0$ and we can take $C_{p}=0$ in this case. 
\end{proof}

Lemma 2.4 of \cite{RHAS} may be used directly and so may the proof. This is because like $\Z$, $\Zi$ has finitely many units. 
\begin{lemma}
There exists a finite set $Q$ of non-zero elements of $\Qi$ with the following property. Suppose $(u,t) \in S_1(X)$. Then we can write $t=a/b^m$, where $a$ and $b$ are Gaussian integers such that gcd$(a,b^m)$ is not divisible by any $m$th power of a non-unit, and $u=qb^n$, with $q \in Q $.
\end{lemma}

\begin{proof}
Because $\mathbb{Z}[i]$ has unique factorization, we have a unique expression $t=a/b^m$ up to multiplication by a unit, with gcd$(a,b^m)$ not divisible by any $m$th power of a non-unit and $a,b \in \mathbb{Z}[i]$. 

We split into two cases. First, assume that $\bp \mid b$. Then, $-\text{val}_{\bp}(t) = m\text{val}_{\bp}(b)-k$, where $0 \leq k < m $, $k \in \Z$. If $m=1$ then $0 \leq k < 1 $, and since $k$ is integral, $k=0$. If $n=1$ then $ 0 \leq k < m $ becomes $ 0 \leq nk < m$, so $0 \leq \frac{n}{m}k < 1 $. In both cases, $ \lceil - \frac{nk}{m}\rceil = 0$. From Lemma 4.5, we have that  $\text{val}_{\bp}(u) = \epsilon + \lceil - \frac{n}{m} \text{val}_{\bp}(t) \rceil = \epsilon + n \text{val}_{\bp}(b) $, with $\left| \epsilon \right| \leq C_{\bp}$. 

Second, if $ \bp \nmid b$ then $\text{val}_{\bp}(t) \geq 0 $, and so $ \left|\text{val}_{\bp}(u) \right| \leq C_{\bp}$. We thus see that $\left|\text{val}_{\bp}(u/b^n) \right| \leq C_{\bp}$ for $\bp$ by hypothesis of this case. 

Note that from Lemma 4.5, in each case we can take $C_{\bp} = 0$ for all $N(\bp) \gg 0$. This implies that for primes $\bp$ with large enough norm, $\text{val}_{\bp}(u/b^n) = \text{val}_{\bp}(q) = 0 $. Taking units into account, there are 4 representations for a given value of $q$ since there are 4 units in $\mathbb{Z}[i]$. This deduction implies that there are finitely many possibilities for $u/b^n$.
\end{proof}

Lemma 2.5 of \cite{RHAS} may be used directly. For the proof, let the maximum value of $q \in Q$ be the element with the largest norm, which yields an appropriate proof.

\begin{lemma}
Suppose $(u,t) \in S_1(X)$ and write $t=a/b^m$ and $u=qb^n$ as in Lemma 4.6. Then $N(a) \leq C_1X^{m/12n}$ and $N(b) \leq C_2X^{1/12n}$, for some constants $C_1$ and $C_2$.
\end{lemma}

\begin{proof}
Write $t=a/b^m$ and $u=qb^n$ as above. The inequality max$(N(A)^3, N(B)^2) < X $ gives us
\begin{center}
$N(u) \cdot \text{max}(N(f(t))^{1/4}, N(g(t))^{1/6}) < X^{1/12}\text{.}$
\end{center}
Let $K_1 > 0$ be a constant such that  $\text{max}(N(f(t))^{1/4}, N(g(t))^{1/6}) \geq K_1$ for all $t$, which exists by Lemma 4.4. For instance, one could take the crude estimate $c_{\infty}^{1/6}$ as a lower bound, where $c_{\bp}$ is as in Lemma 4.4. Then $N(u) \leq K_1^{-1}X^{1/12}$, and so
\begin{center}
$N(b) \leq K_2X^{1/12n}$,
\end{center}
with $K_2 = \left|K_1^{-1/n} \right|  \text{max}_{q \in Q} (\left|q^{-1/n}\right|)$, where the maximal element in $Q$ is the one with greatest norm. Note that while there may be multiple such elements, all of those elements differ by multiplication by a unit, and hence have the same norm. This implies the existence of $K_2$. 

Now suppose that $N(t) \geq 1 $. Let $ K_3 > 0 $ be a constant so that $K_3^4N(t)^r \leq N(f(t)) $ and $ K_3^6N(t)^s \leq g(t)$ holds for all such $t$. This exists since the absolute values of $f(t)/t^r$ and $g(t)/t^s$ will be at least the leading coefficients of $f$ and $g$, respectively. By algebraic manipulation, we have
\begin{center}
$X^{1/12} > N(u)\cdot \text{max}(N(f(t))^{1/4}, N(g(t))^{1/6}) \geq K_3N(u) \cdot \text{max}(N(t)^{r/4}, N(t)^{s/6}) = K_3N(u)N(t)^{n/m} = \left|K_3qa^{n/m}\right|$,
\end{center}
We therefore find $N(a) < K_4X^{m/12n}$ with $K_4=\text{max}_{q \in Q}(\left|(K_3q)^{-m/n}\right|)$. Now suppose $N(t) < 1$. Then $NaA) < N(b^m) \leq K_2^mX^{m/12n}$. Thus, in all cases, we have
\begin{center}
$N(a) \leq K_5X^{m/12n}$,
\end{center}
with $K_5=\text{max}(K_4, K_2^m)$. 
\end{proof}

\begin{corollary}
$\left| N_G(X) \right| \leq K_1X^{(m+1)/12n}$ for some positive constant $K_1$.
\end{corollary}
\begin{proof}
We have that $\left| S_1(X) \right| \leq K_1X^{(m+1)/12n}$ for a positive constant $K_1$ (e.g., $K_1 = K_2K_6$). This implies the same for $S(X)$ (every pair $(A,B)$ has an associated pair $(u,t)$ where $A = u^4f(t)$ and $B = u^6g(t)$), and, by previous deductions, $N_G(X)$.
\end{proof}

\subsubsection{The Lower Bound}
For a pair $(a,b) \in \Zi^2$, set $u = b^n$, $t = a/b^m$, and further set $A=u^4f(t)$ and $B = u^6g(t)$. Let $S_2(X)$ be the number of pairs $(a,b) \in \Zi^2$ satisfying the following:
\begin{itemize}
    \item $a$ and $b$ are coprime.
    \item $N(a) < \kappa X^{m/12n}$ and $N(b) < \kappa X^{1/12n}$.
    \item $4A^3+27B^2 \neq 0$.
\end{itemize}
Fix $\kappa > 0 $ such that if $(a,b) \in S_2(X)$, then $N(A) < X^{1/3} $ and $N(B) < X^{1/2}$. 
\begin{remark}
Such a $\kappa$ exists because $N(f(t)), N(g(t) \leq M $ (the bounds on $a$ and $b$ would $f$ and $g$ defined in a compact region in $\mathbb{C})$. Note as $\kappa \to 0$, $\sup_{N(b) < \kappa X^{1/12n}} b^n \to 0$. Hence, we may choose $\kappa$ small enough such that we satisfy the above conditions.
\end{remark}

The statement of Lemma 2.6 of \cite{RHAS} is identical for $\Qi$ and replacing the condition $p \gg 0 $ with $\left| \bp \right| \gg 0$ yields an appropriate proof. 

\begin{lemma}
There exists a non-zero Gaussian integer $D$ with the following property: if $(a,b) \in \mathbb{Z}[i] $, then gcd$(A^3, B^2)$ divides $D$.
\end{lemma}

\begin{proof}
We find constants $e_{\bp}$ such that $\text{val}_{\bp}(\text{gcd}(A^3, B^2)) \leq e_{\bp}$ and $e_{\bp} = 0$ for $N(\bp) \gg 0 $. This would mean we can take $D= \prod_{\bp}\bp^{e_{\bp}}$, where the product is over all primes $\bp$ in $\mathbb{Z}[i]$. By construction, $D$ satisifies the desired conditions.

We first consider the case where $t$ satisfies Case I of Lemma 4.5. Suppose $(a,b) \in S_2(X) $, and let $\bp$ be a prime. Let $K_1$ be a constant such that 
\\
$\left|3\text{val}_{\bp}(f(t)) - 3r \text{val}_{\bp}(t) \right| \leq K_1$ and $\left|2\text{val}_{\bp}(g(t)) - 2s \text{val}_{\bp}(t) \right| \leq K_1$ for $t \in \mathbb{Q}[i] $ with $\text{val}_{\bp}(t) < 0 $, which exists by the discussion in the proof of Lemma 4.5. Note that for $N(\bp) \gg 0$, we can take $K_1=0$. Suppose that $\text{val}_{\bp}(b) = k > 0$ so that $\text{val}_{\bp}(t)=-mk$ and $\text{val}_{\bp}(u) = nk $. Then,
\begin{center}
$\text{val}_{\bp}(A^3) = \text{val}_{\bp}(u^{12}f(t)^3) = 12nk-3rmk+\epsilon = 12m(\frac{n}{m}-\frac{r}{4})k+\epsilon$
\end{center}
and
\begin{center}
$\text{val}_{\bp}(B^2) = \text{val}_{\bp}(u^{12}g(t)^2) = 12nk-2smk+\delta = 12m(\frac{n}{m}-\frac{s}{6})k+\delta$
\end{center}
where $\left|\epsilon\right| \leq K_1$ and $\left|\delta\right| \leq K_1$. However, $\frac{n}{m}=\text{max}(\frac{r}{4}, \frac{s}{6})$, and so we have $\text{val}_{\bp}(\text{gcd}(A^3, B^2)) \leq K_1 $. We may take $K_1 = 0$ for $\bp$ of large enough norm.

We now handle Case II where $\text{val}_{\bp}(t) < 0$. Without loss of generality, suppose $t \in B(\alpha_i, \delta)$ for some $i$ and $\delta > 0$. Because $A = u^4f(t)$ and $B = u^6g(t)$ are continuous in $t$, $A$ and $B$ are bounded in the closure of $B(\alpha_i, \delta)$. Hence, there exists $K_2$ such that $\text{val}_{\bp}(\gcd (A^3, B^2)) \leq K_2$. For $\bp$ of large enough norm, $\text{val}_{\bp} g(t) \leq 0$ since the coefficients of $g$ will be $\bp$-adic units.  Hence, we may take $\text{val}_{\bp}(u) = 0$, and thus $\text{val}_{\bp}(B^2) = 0$. Thus, we may take $K_2 = 0$ for $\bp$ of large enough norm.

Now suppose $\text{val}_{\bp}(t) \geq 0 $. Let $K_2$ be a constant so that $\text{max}(3\text{val}_{\bp}(f(t)), 2\text{val}_{\bp}(g(t))) \leq K_2$ for all $t \in \mathbb{Q}[i] $ with $\text{val}_{\bp}(t) \geq 0 $. This constant exists by Lemma 4.4. Moreover, we can take $K_2=0$ for sufficiently large $\bp$. Since $ \bp \nmid b $, this gives $\text{val}_{\bp}(\text{gcd}(A^3, B^2)) \leq K_3 $. Then we may take $e_{\bp} = \max (K_1, K_2, K_3)$, and this is 0 for $\bp$ of large enough norm.

\end{proof}
For the next two lemmas we introduce another set. Let $S_3(X)$ denote the set of pairs $(A,B) \in \mathbb{Z}[i]^2 $ coming from $S_2(X)$. We have a map $S_3(X) \rightarrow S(X) $ which sends $(A,B)$ to $(A/d^4, B/d^6)$, where $d^{12}$ is the largest 12th power dividing $\text{gcd}(A^3, B^2)$. Then, the statement of \cite[Lemma 2.7]{RHAS} may also remain unchanged, and letting $d \in \Zi$ provides sufficient justification.

\begin{lemma}
There exists a finite positive constant $N$ such that every fiber of the map $S_3(X) \rightarrow S(X)$ has cardinality at most $N$.
\end{lemma}

\begin{proof}
Let $(A_0, B_0) \in S(X) $ be given. Suppose $(A,B) \in S_3(X) $ maps to $(A_0, B_0) $. That is, $A=d^4A_0$ and $B=d^6B_0$ for some $d \in \mathbb{Z}[i] $. By Lemma 4.8, $d^{12} \mid D$. Therefore, if $N$ is the number of 12th powers dividing $D$, the fibers have cardinality at most $N$.
\end{proof}

The statement and proof of Lemma 2.8 hold over $\Qi$.

\begin{lemma}
There exists a positive finite constant $M$ such that every fiber of the map $S_2(X) \rightarrow S_3(X)$ has cardinality at most $M$.
\end{lemma}

\begin{proof}
Let $(A,B) \in S_3(X)$ be given. An element $(a,b) \in S_2(X)$ of the fiber gives a solution to the equations $Ax^4 = f(y)$ and $Bx^6 = g(y)$ by $x=u^{-1}=b^{-n}$ and $y=t=a/b^m$. These two equations define algebraic plane curves of degree $\text{max}(4,r)$ and $\text{max}(6, s) $, which have no common components since $f$ and $g$ are coprime. By B\'ezout's theorem, they have at most $M=\text{max}(4,r) \text{max}(6,s) $ points in common, which is finite and positive. This bounds the cardinality of the fibers.
\end{proof}

\begin{lemma}
The probability $P$ that two Gaussian integers $a$ and $b$ are coprime exists.
\end{lemma}

\begin{proof}
This probability is equal to $1/\zeta_{\Qi}(2)$, where $\zeta_{\Qi}(2)$ is the Dedekind zeta function of $\Qi$ evaluated at 2. \cite{GCJJ}
\end{proof}
It remains to use the conditions for elements of $S_2(X)$, which we use here.
\begin{corollary}
$K_1X^{(m+1)/12n} \leq \left| N_G(X) \right|$ for some positive constant $K_1$.
\end{corollary}
\begin{proof}
We first make use of the condition that $N(a) < \kappa X^{(m+1)/12n}$ and $N(b) < \kappa X^{1/12n}$. There are at most $\kappa X^{(m+1)/12n}$ possible norms for $a$ and at most $\kappa X^{1/12n}$ possible norms for $b$. Thus, there are at most $\kappa^2 X^{(m+1)/12n} +X^{m/12n} + X^{1/12n} + 1$ pairs $(N(a), N(b)$ that satisfy this condition. Because there are finite number of Gaussian integers with fixed norm, there exists a positive constant $C_1$ such that there are at most $C_1(\kappa^2 X^{(m+1)/12n} +X^{m/12n} + X^{1/12n}) + O(1)$ pairs $(a,b)$ in $S_2(X)$. Then, by Lemma 4.11 we know that $\left| S_2(X) \right| \geq P \cdot (C_1X^{(m+1)/12n} +X^{m/12n} + X^{1/12n})$. Since $\lim_{X \to \infty} \frac{X^{\alpha}}{X^{(m+1)/12n}} = 0$ for $0 < \alpha < \frac{m+1}{12n}$, as $X \to \infty$ we have that $C_2X^{(m+1)/12n} \leq \left| S_2(X) \right|$ for some positive cosntant $C_2$.

Finally, we use the condition on $A$ and $B$ that $4A^3 + 27B^2 \neq 0 $. We exclude such $ (a,b)$ that give $ 4A^3 + 27B^2 = 0 $. Note that this equation gives us an algebraic curve in $\mathbb{Q}[i]^2$, whereas $ (a,b) \in \mathbb{Q}[i]^2 $. The plane $\Qi ^ 2$ has greater dimension than the algebraic curve $ 4A^3 + 27B^2 = 0 $. Thus, as $X \to \infty$, the number of ordered pairs $(A,B)$ excluded will be of lower order. We may conclude that $C_4X^{(m+1)/12n} \leq \left| S_2(X) \right|$ for some positive constant $C_4$. Since the map $ S_2(X) \rightarrow S(X) $ has fibers of bounded size,  $MNC_4X^{(m+1)/12n} \leq \left| S(X) \right|$. Thus, the lower bound has been established and Proposition 4.2 has been proved. Hence we have proven Theorem 2.3 for the groups listed in Case 1, $G = \Z/M\Z, \Z/2\Z \times \Z/2K\Z, \Z/4\Z \times \Z/4\Z$, $M \in [4,10]$, $M=12$, $2 \leq K \leq 4$, with $d(G) = 12n/(m+1)$.
\end{proof}
From the definition of $S_2(X)$, we know it contains $C_1\kappa^2X^{(m+1)/12n} + O(\kappa) $ elements, with  $C_1$ a positive constant and $ O(\kappa)$ an error term of lower magnitude that accounts for possible ordered pairs $(a,b)$ on the boundary of the relevant region in $\mathbb{Z}[i]^2$. We now use the condition that $a$ and $b$ are coprime.

We have that $\left| S_2(X) \right| \geq P(C_2X^{(m+1)/12n} + O(\kappa)) $. Since the error term has much lower order compared to $\kappa^2X^{(m+1)/12n}$, $C_1X^{(m+1)/12n} \leq \left| S_2(X) \right|$ as $X \rightarrow \infty $, for some positive constants $C_2$ and $C_3$.


\subsection{Case 2}
Recall the groups that will be dealt with in this case: $G = \Z/M\Z$, $M = 13, 16, 18$, is parameterized by $\mathcal{E}_{s,t}$ with $(s,t) \in C(\Qi)$, where $C$ is a plane curve of genus $>1$. We have the following data for the parameter spaces of the universal equations of groups $G$ in this case, following from \cite[\S 4]{FPR} and some manipulation:
\begin{table}[H]
\centering
    \begin{tabular}{ | l | l |}
   \hline
    $G$ & $ C(s,t) $ \\ \hline
    $\Z/13\Z $ & $s^2  = t^6-2t^5+t^4-2t^3+6t^2-4t+1$ \\ \hline
    $\Z/16\Z $ & $s^2  = t(t^2+1)(t^2+2t-1)$ \\ \hline
    $\Z/18\Z $ & $s^2 = t^6+2t^5+5t^4+10t^3+10t^2+4t+)$ \\ \hline
    \end{tabular}
\end{table}
The curves $C$ that form each parameter space have genus greater than 1. Furthermore the curves are irreducible. Thus, by Faltings' Theorem (\cite{Faltings1983}) there are a finite number of points $s,t \in \Qi$ such that $(s,t) \in C(\Qi)$, and so for each group $G$ in this case, there are finitely many curves that have torsion subgroup $G$. Therefore, we have that $d(G) = \infty$.
\subsection{Case 3}
Recall the groups that will be dealt with in this section: $G = \Z/M\Z, \Z/2\Z \times \Z/10\Z, \Z/2\Z \times \Z/12\Z $, $M = 11, 14, 15 $. For the groups $G$ that fall into this case, the parameter spaces $C(s,t) = 0$ are elliptic curves. We have the following data for the Mordell-Weil groups of these curves $C$, following from \cite[\S 4]{FPR} and some manipulation:
\begin{table}[H]
\centering
    \begin{tabular}{ | l | l |}
   \hline
    $G$ & Mordell-Weil group for $C(\Qi)$ \\ \hline
    $\Z/11\Z $ & $\Z/5\Z$ \\ \hline
    $\Z/14\Z $ & $\Z/6\Z$ \\ \hline
    $\Z/15\Z $ & $\Z/4\Z$ \\ \hline
    $\Z/2\Z \times \Z/10\Z$ & $\Z/6\Z$ \\ \hline
    $\Z/2\Z \times \Z/12\Z$ & $\Z/8\Z$ \\ \hline
    \end{tabular}
\end{table}
This data clearly demonstrates that for each group $G$ in this case, the corresponding group $C(\Qi)$ is finite, and thus there are a finite number of elliptic curves with torsion subgroup isomorphic to $G$. This establishes $d(G) = \infty$ for this case.

\section{The Group $\Z/2\Z \times \Z/2\Z$}

For the group $\Z/2\Z \times \Z/2\Z$, we may begin with an analogue of \cite[Proposition 4.2]{RHAS}:
\begin{prop}
Let $ f, g \in \mathbb{Q}[t] $ be coprime polynomials of degrees $ r $ and $s$. Assume at least one of $r$ or $s$ is positive. Write 
\begin{center} 
max$(\frac{r}{4}, \frac{s}{6}) = \frac{n}{m}$
\end{center}
with $n$ and $m$ coprime. Assume $n=1$ or $m=1$. Let $S(X)$ be the set of pairs $(A,B) \in \mathbb{Z}[i]^2$ satisfying the following conditions: 

\begin{itemize}
\item $4A^3+27B^2 \neq 0$.
\item $\text{gcd}(A^3, B^2)$ is not divisible by any 12th power of a non-unit.
\item $N(A) < X^{1/3}$ and $N(B) < X^{1/2}$.
\item There exist $ u, t \in \mathbb{Q}[i]$ such that $A=u^2f(t)$ and $B=u^3g(t)$.
\end{itemize}
Assume $m+1>n$. Then $K_1X^{(m+1)/6n} \leq |S(X)| \leq K_2X^{(m+1)/6n}$ for some positive constants $K_1$ and $K_2$.
\end{prop}
\begin{proof}
The proof is similar to the proof of Proposition 4.2. A version of Lemma 4.5 holds, but $C_{\bp} = 1$ for $N(\bp) \gg 0$. A version of Lemma 4.6 holds where we can write $t = a/b^m$ with gcd$(a,b^m)$ not divisible by any $m$th power of a non-unit and $u = qcb^n$ where $q$ belongs to a finite set and $c$ is square-free. The analog of Lemma 4.7 yields $|ca^{m/n}| \leq K_3X^{1/6}$ and $|cb^n| \leq K_4X^{1/6}$ for some positive constants $K_3$ and $K_4$. For any given $c$, the number of possibilities for $a$ is $O(\frac{X^{m/6n}}{c^{m/n}})$ and the number of possibilities for $b$ is $O(\frac{X^{1/6n}}{c^{1/n}})$, yielding $O(\frac{X^{(m+1)/6n}}{c^{(m+1)/n}})$ total possibilities. Integrate over $c$, up to $X^{1/6}$, to get $|S(X)| \leq K_2X^{(m+1)/6n}$ for some positive constant $K_2$.

For the lower bound, fix $c = 1$. Let $\kappa > 0$ be a small constant, and consider the set $S = \{ (a,b) \in \Zi^2: |a^{n/m}|<\kappa X^{1/6}, |b^n| < \kappa X^{1/6}, \Delta \neq 0 \}$. Then, $|S(X)| \leq K_4X^{(m+1)/6n}$ for some positive constant $K_4$. The analog of Lemma 4.8 says that gcd$(A^3, B^2) = \alpha^6 \beta$, where $\beta \mid D$ for some fixed $D \in \Zi$, and furthermore $\alpha$ is square-free. We can then apply Lemmas 4.9-4.11 to obtain the lower bound.
\end{proof}
We may deduce from the parameterizations given in \cite{FPR} that the polynomials of the universal family for $\Z/2\Z \times \Z/2\Z$ are $f(t) = \frac{1}{3}(t^2-t+1)$ and $g(t) = \frac{1}{27}(-2t^3+3t^2+3t-2)$. Then, $m = n = 1$ and so $(m+1)/6n = 1/3$. From this we obtain that $d(\Z/2\Z \times \Z/2\Z) = 3$.

  
\appendix

\section{Computation of Degrees of $f$ and $g$ in \S 4}
We use the parameterizations given in \cite{FPR}, which gives parameterizations in the form $y^2+axy+by=x^3+dx^2$. The parameters $a$ and $b$ are given by the quotients of polynomials. Let $a=\frac{a(t)}{c_1(t)}$, $b = \frac{b(t)}{c_2(t)}$, and $d=\frac{d(t)}{c_3(t)}$. We make a change of variable $(x,y) \mapsto (\frac{x}{u(x)^2}, \frac{y}{u(x)^3})$ to clear the denominators of $a$, $b$ and $d$. Denote these transformed polynomials as $a_1(t)$, $b_1(t)$, and $d_1(t)$. We now convert this form into Weierstrass form and compute the degree. 

After transforming this given into the desired form, we compute the maximum of the degree of each term in the expression for the polynomials $f$ and $g$. From here, let $a = a_1(t)$, $b = b_1(t)$, and $d = d_1(t)$.

The transformations are as follows:
\[ y^2+axy+by = x^3+dx^2 \]
\[\implies (y+ \frac{ax+b}{2})^2 = x^3+dx^2 + (\frac{ax+b}{2})^2. \]
Let $ y + \frac{ax+b}{2} \mapsto y $. Then,
\[ y^2 = x^3+(\frac{a^2}{4} + d)x^2  + \frac{abx}{2} + \frac{b^2}{4}.\]
\[ \implies y^2 = (x+ (\frac{a^2}{12} + \frac{d}{3}))^3 -\frac{d^2x}{3}-\frac{a^4x}{48}-\frac{a^2dx}{6}+\frac{abx}{2}-\frac{a^6}{1728}-\frac{a^4d}{144}-\frac{a^2d^2}{36}-\frac{d^3}{27}+\frac{b^2}{4}. \]

\noindent Let $ x+ \frac{a^2}{12}+\frac{d}{3} \mapsto x$. Then,
\[y^2 = x^3 - (\frac{d^2}{3}+\frac{a^4}{48} + \frac{a^2d}{6}-\frac{ab}{2})(x-\frac{d}{3}-\frac{a^2}{12})-\frac{a^6}{1728}-\frac{a^4d}{144}-\frac{a^2b^2}{36}-\frac{b^3}{27}+\frac{b^2}{4} \]
\[ \implies y^2=x^3-(\frac{d^2}{3}+\frac{a^4}{48}+\frac{a^2d}{6}-\frac{ab}{2})x + (\frac{d}{3} + \frac{a^2}{12})(\frac{d^2}{3}+\frac{a^4}{48}+\frac{a^2d}{6}-\frac{ab}{2})-\frac{a^6}{1728}-\frac{a^4d}{144}-\frac{a^2d^2}{36}-\frac{d^3}{27}+\frac{b^2}{4}.\]
We now let $f(t)$ be the coefficient of $x$ and $g(t)$ the constant term. The degrees of $f$ and $g$ were found using MATLAB.
\bibliographystyle{amsalpha}
\bibliography{sources.bib}

\newcommand{\etalchar}[1]{$^{#1}$}
\providecommand{\bysame}{\leavevmode\hbox to3em{\hrulefill}\thinspace}
\providecommand{\MR}{\relax\ifhmode\unskip\space\fi MR }
\providecommand{\MRhref}[2]{%
  \href{http://www.ams.org/mathscinet-getitem?mr=#1}{#2}
}
\providecommand{\href}[2]{#2}
\begin{thebibliography}{Kim01}

\bibitem[CJ88]{GCJJ}
George~E Collins and Jeremy~R Johnson, \emph{{T}he {P}robability of {R}elative
  {P}rimality of {G}aussian {I}ntegers}, International Symposium on Symbolic
  and Algebraic Computation, Springer, 1988, pp.~252--258.

\bibitem[Fal83]{Faltings1983}
G.~Faltings, \emph{Endlichkeitssätze für abelsche varietäten über
  zahlkörpern.}, Inventiones mathematicae \textbf{73} (1983), 349--366 (ger).

\bibitem[HS13]{RHAS}
Robert Harron and Andrew Snowden, \emph{Counting {E}lliptic {C}urves with
  {P}rescribed {T}orsion}, Crelles Journal (2013), 1--16.

\bibitem[K{\etalchar{+}}86]{KM}
Sheldon Kamienny et~al., \emph{Torsion {P}oints on {E}lliptic {C}urves {O}ver
  {A}ll {Q}uadratic {F}ields}, Duke Mathematical Journal \textbf{53} (1986),
  no.~1, 157--162.

\bibitem[Kim01]{IK}
Ian Kiming, \emph{Group law for elliptic curves}, 2000-2001.

\bibitem[KM88]{KUMO}
Monsur~A Kenku and Fumiyuki Momose, \emph{Torsion {P}oints on {E}lliptic
  {C}urves {D}efined {O}ver {Q}uadratic {F}ields}, Nagoya Mathematical Journal
  \textbf{109} (1988), 125--149.

\bibitem[Rab10]{FPR}
Patrick Rabarison, \emph{Structure de {T}orsion des {C}ourbes {E}lliptiques sur
  les {C}orps {Q}uadratiques}, Acta Arithmetica (2010), 17--52.

\end{thebibliography}
\end{document}